\documentclass[12pt]{amsart}
\usepackage{amssymb,amsthm, amsmath}
\usepackage{fullpage,hyperref}
\usepackage{color}

\newtheorem{theorem}{Theorem}[section]

\newtheorem{lemma}[theorem]{Lemma}
\theoremstyle{definition}
\newtheorem{definition}[theorem]{Definition}
\newtheorem{remark}[theorem]{Remark}

\newcommand{\R}{\mathbb R}

\newcommand{\RNum}[1]{\uppercase\expandafter{\romannumeral #1\relax}}

\definecolor{piros}{rgb}{.8,0,0}
\definecolor{kek}{rgb}{0,0,.8}
\definecolor{purple}{rgb}{.4,0,.4}

\title{On the Existence of Ordinary Triangles}
\author[Fulek, Mojarrad, Nasz\'odi, Solymosi, Stich, Szedl{\'a}k]{Radoslav 
Fulek 
\and Hossein Nassajian Mojarrad \and M\'arton Nasz\'odi \and J\'ozsef Solymosi 
\and Sebastian U. Stich \and May Szedl{\'a}k}
\date{}


\begin{document}
\begin{abstract}
Let $P$ be a finite point set in the plane.
A \emph{$c$-ordinary triangle} in $P$ is a subset of $P$ consisting of three 
non-collinear points such that each of the three lines determined by the three 
points contains at most $c$ points of $P$.
Motivated by a question of Erd\H{o}s, and answering a question of de Zeeuw, we 
prove that there exists a constant $c>0$ such that  $P$ contains a $c$-ordinary 
triangle, provided that $P$ is not contained in the union of two lines. 
Furthermore, the number of $c$-ordinary triangles in $P$ is $\Omega(|P|)$.
\end{abstract}
\maketitle
\section{Introduction}

In 1893, Sylvester~\cite{Syl93} asked whether, for any finite set of 
non-collinear points on the Euclidean plane, there exists a line incident 
with exactly two points.
The positive answer to this question, now known as the \emph{Sylvester--Gallai 
theorem}, was first obtained almost half a century later in 1941 by 
Melchior~\cite{Mel41} as a consequence of the positive answer to an analogous 
question in the projective dual.
Erd\H{o}s~\cite{Erd43}, unaware of these developments, posed the same 
problem in 1943, and it was solved by Gallai in 1944.
For more on the history of this and related problems, see \cite{GT13}.

Given a finite set of points $P$ on the Euclidean plane, a line $\ell\subset 
\R^2$ is \emph{determined} by $P$ if $\ell$ contains at least two 
points of $P$. We say that $\ell$ is an \emph{ordinary line}, if $\ell$ 
contains exactly two points of $P$.
Erd\H{o}s \cite{E84} considered the problem of finding an ordinary triangle, 
that is, three ordinary lines determined by three points of a finite planar 
point set. See \cite{Bor90} for details on the origin of this problem.
Motivated by this problem, and with an application in studying 
ordinary conics \cite{BZV16}, de Zeeuw asked a related question at the 13th 
Gremo's Workshop on Open Problems (GWOP 2015, Feldis, Switzerland), which we 
describe below.

\begin{definition}[$c$-ordinary triangle]
Let $c$ be a natural number and  let $P$ be a point set in the plane. 
A $c$-\textit{ordinary triangle} in $P$ is 
a subset of $P$ consisting of three non-collinear points such that each of the 
three lines determined by the three points contains at most $c$ points of $P$.
\end{definition}

It is easy to see that in order to be able to find a $c$-ordinary triangle 
for large $n$, we have to assume that $P$ is not contained in the union of two 
lines.
Under this restriction one might suspect that there is a 2-ordinary triangle in 
$P$. However, this is not true as shown by B\"or\"oczky's 
construction~\cite[Figure 4,5,6]{GT13}.
The following simple example also shows this. Let $P_1$ be a set of points 
that are not all collinear and let $\ell$ be some line. For each line $\ell_1$ 
determined by the point set $P_1$, we add the point at the intersection of 
$\ell$ and $\ell_1$. Let us denote this new set of points by $P_2$. All points 
of $P_2$ are collinear, hence a $2$-ordinary triangle must contain two points from 
$P_1$. However, by construction every line determined by $P_1$ contains a point 
of $P_2$. Hence there are no 2-ordinary triangles in this point set. 

De Zeeuw asked whether a $c$-ordinary triangle can be found in $P$. 
The aim of this manuscript is to give a positive answer to this question.
\begin{theorem}\label{main}
There is a natural number $c$ such that the following holds. Assume $P$ is a 
finite set of points on the Euclidean plane not contained in the union of two 
lines. Then $P$ contains a $c$-ordinary triangle, that is three non-collinear 
points such that each of the three lines determined by these three points 
contains at most $c$ points of $P$.  Moreover, the number of $c$-ordinary 
triangles in $P$ is $\Omega(|P|)$.
\end{theorem}

\begin{remark}\label{rem:c}
The constant in the theorem above can be chosen as $c=12000$.
\end{remark}
We see no reason to believe that this is the best constant.
Moreover, it remains open if the number of  $c$-ordinary 
triangles in $P$ is superlinear (possibly even quadratic) in $|P|$.

\section{Tools}
To prove Theorem~\ref{main}, we need the following lemmas. The first one is a 
corollary of the Szemer\'{e}di-Trotter Theorem. 

\begin{lemma}\label{ST}\cite{ST,PRTT06}
Let $k,n \ge 2$ be natural numbers, $P$ a set of $n$ points in the plane, and 
let $f(k)$ denote the number of lines in the plane containing at least $k$ 
points of $P$. Then 
\[
f(k)\le\begin{cases}
c' \frac{n^2}{k^3}, \mbox{if } k \le \sqrt{n},\\
c' \frac{n}{k}, \mbox{if } k > \sqrt{n}              
\end{cases}
\]
for a universal constant $c'>0$.
In fact, we may take $c' = 125$.
\end{lemma}

\begin{proof}
Clearly, the claimed bound holds for $k=2,3$, since $f(2)\le {\binom{n}{2}}$ 
and $f(3)\le {\binom{n}{2}}/{\binom{3}{2}}$. To prove the statement for $k > 
3$, we rely on the following result by Pach, Radoi{\v c}i\'c, Tardos and 
T\'oth \cite[Corollary 5.1]{PRTT06}: for any given $n$ points and $m$ lines on 
the Euclidean plane, the number of incidences between them is at most 
$2.5m^{2/3} n^{2/3} + m + n$. Let $m=f(k)$ denote the number of lines 
containing at least $k$ points of $P$. Observe that the number of point-line 
incidences are thus at least $mk$. Hence, $mk \leq 2.5m^{2/3} n^{2/3} + 
m + n$.

First, consider the case $m \geq n$. Observe that for $k>3$, we have $mk/2 \leq 
m(k-2) \leq 2.5m^{2/3} n^{2/3}$.
It follows that $m \leq 125 \frac{n^2}{k^3}$, and specifically, $m \leq 125 
\frac{n}{k}$ if $k > \sqrt{n}$.

Next, consider the case $m < n$. We have $mk \leq 2.5m^{2/3} n^{2/3} + m + n 
\leq 2.5m^{2/3} n^{2/3} +  2.5 n$, and therefore $mk \leq \max\{5 m^{2/3} 
n^{2/3},5n\}$. Hence, $m \leq \max\{ 125 \frac{n^2}{k^3}, 5\frac{n}{k}\}$. For 
$k \leq 5\sqrt{n}$, the maximum is attained at the first term, whereas for $k > 
\sqrt{n}$, we trivially have $\frac{n^2}{k^3} < \frac{n}{k}$,
establishing the claim for $c'=125$.
\end{proof}

The following Tur{\'a}n--type lemma (related to Mantel's theorem) from extremal 
graph theory provides a lower bound for the number of triangles (subgraphs 
isomorphic to $K_3$) in a graph.  It can be found with a proof as Problem~10.33 
in \cite{Lov07}.
\begin{lemma}\label{EG}
Consider a graph $G=(V(G), E(G))$ with $|V(G)|=n$ and $|E(G)|=m$. Let 
$t_3(G)$ denote the number of triangles in $G$. Then we have
$$t_3(G) \ge \frac{m}{3n}(4m-n^2).$$
\end{lemma}

\section{Proof of Theorem~\ref{main}}

In this section, we prove Theorem~\ref{main}. 
Our proof is closely related to the standard proof of Beck's Theorem, where the 
number of pairs of points on medium-rich lines is bounded using the 
Szemer\'edi-Trotter theorem, and then it is concluded that either there is a 
very rich line, or there are many pairs of points on poor lines, see the proof 
of Theorem~18.8 in \cite{Ju11}.

The constant $c$ will be chosen at the end of the proof. Assume $P$ is a set of 
$n \ge c$ points in the plane and let $\mathcal{L}=\{L_1,L_2,\dots,L_m\}$ 
denote the set of lines determined by $P$. Define $l_i=|L_i \cap P|$, for  
$i=1,2,\dots,m$. \\ Set $\alpha=\frac{4}{c+1}$. We split the proof into two 
cases:
\begin{itemize}
\item [(i)] There is a line $L_i \in \mathcal{L}$ such that $l_i > \alpha n$;
\item [(ii)] For all $i=1,2,\dots,m$ we have $l_i \le \alpha n$.
\end{itemize} 

Consider the first case. Since the point set $P\setminus L_i$ is non-collinear 
by the assumption, by applying the Sylvester-Gallai theorem, we can 
find an 
ordinary line $L \in \mathcal{L}$ for $P\setminus L_i$, i.e. $L$ contains 
exactly two points $q,r \in P\setminus L_i$. Note that $L$ may contain at most 
one point of $P \cap L_i$. Next, we show that there are many points on $L_i$ 
which 
together with $q,r$ form $c$-ordinary triangles. For this, we define the set 
$P_q \subset P$ as
$$P_q=\{ p \in L_i\cap P : |\overline{pq} \cap P| > c\},$$
where $\overline{pq}$ denotes the line passing through $p,q$. We define $P_r$ 
in 
a similar way. Note that for any point $p\in P_q$, the line $\overline{pq}$ 
contains at least $c-1$ points of $P\setminus ( L_i\cup\{q\})$, moreover, these 
sets of $c-1$ points are disjoint for different $p\in P_q$. So we 
get
$$(c-1)\cdot |P_q| \le n-l_i,$$
which implies that
$$|P_q| \le \frac{n-l_i}{c-1} < 
\frac{l_i/\alpha-l_i}{c-1}=\frac{(c+1)/4-1}{c-1}\cdot l_i<\frac{l_i}{4}.$$
Similarly, $|P_r| < \frac{l_i}{4}$. So there are at 
least $\frac{l_i}{2}$ points 
$s\in P\cap L_i$ such that $s \notin P_q\cup P_r$. Furthermore, $s,q,r$ are 
non-collinear. This implies that the lines $\overline{sq},\overline{sr}$ 
contain 
at most $c$ points of $P$. Therefore every triangle determined by $s,q,r$, 
where $s \notin P_q\cup P_r$, is a 
$c$-ordinary triangle for $P$. 
The number of these triangles is at least
$\frac{l_i}{2}>\frac{\alpha n}{2}=\frac{2n}{c+1}$, 
completing the proof of case~(i). Note that, so far, $c$ may be chosen as any 
integer greater than 2.

Next, we consider case (ii). So we assume that no line of $\mathcal{L}$ 
contains more than $\alpha n$ points of $P$. First we bound $\sum_{c < l_i \le 
\alpha n} \binom{l_i}{2}$ from above. With the notation of Lemma~\ref{ST}, we 
have

\begin{align*}
\sum_{i\;:\;c < l_i \le \sqrt{n}} \binom{l_i}{2} &\le 
\sum_{j=\lfloor\log (c+1)\rfloor}^{\lceil\log{\sqrt{n}}\rceil} \;\;\;\sum_{i: 2^j \le l_i \le 2^{j+1}} \binom{l_i}{2} \le  
\sum_{j=\lfloor\log (c+1)\rfloor}^{\lceil\log{\sqrt{n}}\rceil} f(2^j) \binom{2^{j+1}}{2}  \\
&\overset{*}{\le}  \sum_{j=\lfloor\log (c+1)\rfloor}^{\lceil\log{\sqrt{n}}\rceil} c' \frac{n^2}{2^{3j}} 
\binom{2^{j+1}}{2}  \le  \sum_{j=\lfloor\log (c+1)\rfloor}^{\lceil\log{\sqrt{n}}\rceil} c' 
\frac{n^2}{2^{j-1}}\\
 &\le  \sum_{j=\lfloor\log (c+1)\rfloor}^{\infty} c' \frac{n^2} {2^{j-1}}\leq\frac{8c'n^2}{c+1},
\end{align*}
where logarithms are base $2$, and the inequality with star follows 
from Lemma~\ref{ST}. 

On the other hand, by the same lemma, we have
\begin{align*}
\sum_{i\;:\;\sqrt{n} < l_i \le \alpha n} \binom{l_i}{2} &\le 
\sum_{j=0}^{\lceil\log(\alpha\sqrt{n})\rceil-1} \sum_{2^j\sqrt{n} < l_i \le 
2^{j+1}\sqrt{n}} 
\binom{l_i}{2} \le  \sum_{j=0}^{\lceil\log(\alpha\sqrt{n})\rceil-1} f(2^j\sqrt{n}) 
\binom{2^{j+1}\sqrt{n}}{2} \\
&\overset{*}{\le} \sum_{j=0}^{\lceil\log(\alpha\sqrt{n})\rceil-1} c' 
\frac{n}{2^{j}\sqrt{n}} 
\binom{2^{j+1}\sqrt{n}}{2} \\
&\leq \sum_{j=0}^{\lceil\log(\alpha\sqrt{n})\rceil-1} c'n^{3/2} 2^{j+1} \le 4c'n^{3/2} 
\cdot 
\alpha \sqrt{n}=\frac{16c'n^2}{c+1}.
\end{align*}

As a result, we obtain
$$\sum_{i\;:\;c < l_i \le \alpha n} \binom{l_i}{2}= \sum_{i\;:\;c < l_i \le 
\sqrt{n}} 
\binom{l_i}{2}+\sum_{i\;:\;	\sqrt{n} < l_i \le \alpha n} \binom{l_i}{2} \le 
\frac{24c'n^2}{c+1}.$$

Let $G$ be the graph with vertex set $V(G)=P$, such that two points $p,p' 
\in P$ are adjacent in $G$ if the line $\overline{pp'}$ spanned by $p,p'$ 
satisfies $|\overline{pp'} \cap P| \le c$.

By the following identity
$$\sum_{i\;:\;2 \le l_i \le \alpha n} \binom{l_i}{2}=\binom{n}{2},$$
we obtain for the number of edges of $G$, 
\begin{equation}
|E(G)| =\sum_{i\;:\;2 \le l_i \le c} \binom{l_i}{2} \ge \binom{n}{2}- 
\frac{24c'n^2}{c+1}.
\label{eq6}
\end{equation}

Now we choose $c$ large enough such that 
\begin{equation}
4\left(\binom{n}{2}-  \frac{24c'n^2}{c+1}\right)-n^2=\Omega(n^2). 
\label{eq:determine_c}
\end{equation} Combining 
it with \eqref{eq6} yields $$4|E(G)|-n^2=4\left(\sum_{i:\;2 \le l_i \le c} 
\binom{l_i}{2}\right)-n^2=\Omega(n^2).$$ Therefore, by Lemma~\ref{EG} we have 
$$t_3(G) \ge \frac{|E(G)|}{3n}(4|E(G)|-n^2)=\frac{\Omega(n^2)}{n}\cdot 
\Omega(n^2)=\Omega(n^3).$$
This implies that $G$ has $\Omega(n^3)$ triangles. Let $T$ be the set of those 
triangles in $G$ whose three vertices are non-collinear. It is easy to see that these 
triangles correspond to $c$-ordinary triangles in $P$. 

Note that the number of triangles with collinear vertices is at most
\[\sum_{i\;:\;2 \le l_i \le c} \binom{l_i}{3}\le \sum_{i\;:\;2 \le l_i \le 
c}\binom{c}{3}\le \binom{n}{2}\cdot \binom{c}{3}=O(n^2).\]
So we get
 $$|T|=\Omega(n^3)-O(n^2)=\Omega(n^3).$$
As a result, $P$ has $\Omega(n^3)$ $c$-ordinary 
triangles, provided that $c$ satisfies \eqref{eq:determine_c}. 

\begin{proof}[Proof of Remark~\ref{rem:c}]
Equation~\eqref{eq:determine_c} yields that we may choose $c = 96 c'$, where 
$c'$ is from Lemma~\ref{ST}.
\end{proof}

\section*{Acknowledgements}
We thank Emo Welzl for providing the venue, his GWOP workshop, where our 
research initially started. We also thank Frank de Zeeuw for his many remarks 
on earlier versions of the manuscript.
We are grateful to both referees, whose comments made the presentation much 
cleaner. 

R. Fulek was partially supported by the People Programme (Marie Curie Actions) 
of the European Union's Seventh Framework Programme (FP7/2007-2013) under REA 
grant agreement no [291734].
H. N. Mojarrad and M. Nasz{\'o}di were members of J{\'a}nos Pach's Chair of DCG 
at EPFL, Lausanne, supported by the Swiss National Science 
Foundation (SNSF) Grants 200020-162884 and 200021-165977.
M. Nasz{\'o}di was also partially supported by the National Research, 
Development and Innovation Office grant K119670, and by the 
J\'anos Bolyai Research Scholarship of the Hungarian 
Academy of Sciences.
J. Solymosi was supported by NSERC, by ERC Advanced Research Grant no 267165 
(DISCONV) and by the National Research, Development and Innovation Office 
(NKFIH) grant NK 104183.
S. U. Stich acknowledges support from SNSF and grant ``ARC 14/19-060'' from the 
``Direction de la recherche scientifique - Communaut\'{e} fran\c{c}aise de 
Belgique''. M. Szedl{\'a}k's research was supported by the SNSF Project 
200021\_150055/1.

\bibliographystyle{amsalpha}
\bibliography{biblio} 
\end{document}